\documentclass[11pt]{article}

\usepackage[margin=1in]{geometry}
\usepackage{amsmath,amssymb,amsthm}
\usepackage{mathtools}
\usepackage{tikz}
\usepackage{hyperref}

\newtheorem{theorem}{Theorem}[section]
\newtheorem{lemma}[theorem]{Lemma}
\newtheorem{corollary}[theorem]{Corollary}

\DeclareMathOperator{\diam}{diam}
\DeclareMathOperator{\diag}{diag}
\DeclareMathOperator{\Tr}{Tr}
\DeclareMathOperator{\tr}{tr}

\newcommand{\DL}{D^L}

\title{On a distance Laplacian analog of Brouwer's conjecture \\ for several classes of graphs}
\author{
Silin Huang \\
\small College of Information Science and Technology, Jinan University, Guangzhou, China \\
\small\texttt{huangsilin@stu.jnu.edu.cn}
}
\date{}

\begin{document}
\maketitle

\begin{abstract}
Zhou~et~al.~(2025) proposed a distance Laplacian analog
of Brouwer's conjecture on partial sums of Laplacian eigenvalues,
asserting that for any connected graph $G$,
$\sum_{i=1}^r \partial_i^L(G)\le W(G)+\binom{r+2}{3},$
where $\partial_i^L(G)$ are the eigenvalues
of the distance Laplacian matrix
and $W(G)$ is the Wiener index.
We prove this inequality for three broad classes of graphs,
thereby improving and extending existing results.
First, we prove that all connected graphs
of diameter at most $D$ satisfy the inequality
once the order $n$ satisfies $n\ge\lceil\frac49(D+1)^3\rceil$.
Second, we show that the inequality holds
for every diameter-$2$ graph
with the only exceptions being $K_{1,3}$ at $r=2$
and $K_{1,4}$ at $r=3$.
Third, we prove that if the maximum degree is $\Delta(G)=n-k$,
then the inequality holds for all $n\ge N(k)$,
where $N(2)=10$ and $N(k)=\lceil 5(k-1)^{3/2}\rceil$ for $k\ge 3$.
Our proofs rely on decomposing the distance Laplacian matrix
into Laplacian matrices of auxiliary graphs
whose edges are vertex pairs at distance
at least a prescribed value,
together with classical eigenvalue inequalities.
\end{abstract}

\section{Introduction}
Brouwer's conjecture for graph Laplacians asserts that, if
$\lambda_1\ge\cdots\ge\lambda_n=0$
are the eigenvalues of the Laplacian matrix $L(G)$
of a graph $G(V,E)$ with $|E|=m$ edges, then
\begin{equation}
    S_r\coloneqq\sum_{i=1}^r\lambda_i\le m+\binom{r+1}{2}.
\end{equation}
This conjecture was first introduced by
Brouwer and Haemers
in their monograph~\cite[p.~60]{Brouwer},
and has generated a large body of work.
The cases $r=2$ and $r=3$ were proved by
Haemers~et~al.~\cite{Haemers}
and Wang~et~al.~\cite{Wang}, respectively.
Many further partial results and structural constraints
are also known, including work of Du and Zhou~\cite{Du},
Chen~\cite{Chen2018,Chen2019}
and Ganie~et~al.~\cite{Ganie}.

Independently, the distance Laplacian matrix
was introduced by Aouchiche and Hansen~\cite{Aouchiche2013}
as the Laplacian analog of the distance matrix:
\begin{equation}
    \DL(G)=\diag(\Tr_G(v):v\in V(G))-D(G).
\end{equation}
Distance Laplacian spectra
have since been studied from several perspectives,
including spectral properties,
eigenvalue multiplicities,
extremal eigenvalues, and cospectrality;
see, for example,
\cite{Aouchiche2014,Nath,Lin,Fernandes,Brimkov}
and the survey~\cite{Rather}.

In the setting of distance Laplacian matrices,
a natural analog of Brouwer's conjecture
is to determine an upper bound for
\begin{equation}
    U_r\coloneqq\sum_{i=1}^r \partial_i^L(G),
\end{equation}
where $\partial_1^L(G)\ge\dots\ge\partial_r^L(G)$
are the eigenvalues of $\DL(G)$.
Motivated by the identity $\tr(\DL(G))=2W(G)$,
Zhou~et~al.~\cite{Zhou} proposed the following
distance Laplacian analog of Brouwer's conjecture:
\begin{equation}\label{eq:target}
    U_r=\sum_{i=1}^r \partial_i^L(G) \le W(G)+\binom{r+2}{3},
\end{equation}
where
\begin{equation}
    W(G)=\sum_{\{u,v\}\subseteq V(G)} d_G(u,v)
\end{equation}
is the Wiener index.

Before this conjecture was formulated,
Pirzada and Khan~\cite{Pirzada} already studied
the sum of the largest distance Laplacian eigenvalues
and obtained bounds in terms of
the order, diameter, and Wiener index.
Zhou~et~al.~\cite{Zhou} then proposed this conjecture
and proved it for diameter-$1$ graphs
and for some connected diameter-$2$ graphs
with large maximum degree.

Recently, Mushtaq~et~al.~\cite{Mushtaq}
proved it for several
diameter-$3$ and diameter-$4$ graph families,
including graphs in $\Gamma_n$
(diameter-$3$ graphs
that have a spanning tree of diameter at most $3$)
and graphs with $\Delta=n-2$ for $n\ge 95$,
sun graphs $S(a,a)$ for $a\ge 12$,
and partial sun graphs $S(a,k)$ for $a\ge 450$.

This paper improves and extends some of these results.
In the remainder of this paper,
we refer to the distance Laplacian analog
of Brouwer's conjecture,~\eqref{eq:target},
as the \emph{target inequality}.

The paper is organized as follows.

In Section~\ref{sec:2},
we collect notation and the eigenvalue inequalities
used throughout the paper.

In Section~\ref{sec:3},
we prove that all connected graphs of diameter
at most $D\ge 2$ satisfy the target inequality once
\begin{equation}
  n\ge\left\lceil\frac49(D+1)^3\right\rceil.
\end{equation}

In Section~\ref{sec:4},
we solve the case of diameter-$2$ graphs completely.
We prove that apart from $K_{1,3}$ at $r=2$
and $K_{1,4}$ at $r=3$,
every connected diameter-$2$ graph
satisfies the target inequality.

In Section~\ref{sec:5},
we prove a maximum-degree theorem.  If
$\Delta(G)=n-k$, then the target inequality holds once
$n\ge N(k)$, where $N(2)=10$ and
$N(k)=\lceil 5(k-1)^{3/2}\rceil$
for $k\ge 3$.

\section{Preliminaries}\label{sec:2}
All graphs $G(V,E)$ in this paper are
finite, simple, undirected, and connected,
unless stated otherwise.
For a vertex $u\in V(G)$,
let $\deg_G(u)$ denote its degree in $G$.
For vertices $u,v\in V(G)$,
let $d_G(u,v)$ denote their distance in $G$.
Then, the \emph{distance matrix} of $G$ is $D(G)=(d_G(u,v))$,
the \emph{transmission} of $v$ is $\Tr(v)=\sum_u d_G(u,v)$,
and the \emph{distance Laplacian} matrix of $G$ is
\begin{equation}
    \DL(G)=\diag(\Tr(v):v\in V(G))-D(G).
\end{equation}
The eigenvalues of $\DL(G)$ are denoted by
\begin{equation}
    \partial^L_1(G)\ge\partial^L_2(G)\ge\cdots\ge\partial^L_n(G)=0,
\end{equation}
where the last equality follows from~\cite[Theorem~2.2]{Aouchiche2014}.
The partial sum of these eigenvalues is denoted
\begin{equation}
    U_r(G)=\sum_{i=1}^r \partial^L_i(G),
\end{equation}
while the \emph{Wiener index} of $G$ is denoted
\begin{equation}
    W(G)=\sum_{\{u,v\}\subseteq V(G)} d_G(u,v).
\end{equation}
The target inequality,
which is an analog of the inequality in Brouwer's conjecture, is
\begin{equation}
    U_r(G)\le W(G)+\binom{r+2}{3}\qquad (1\le r\le n).
\end{equation}

We shall use two standard eigenvalue inequalities.

\begin{lemma}[Ky Fan inequality]\label{lem:ky-fan}
    For Hermitian matrices $A,B$ and $1\le r\le n$,
    \begin{equation}
        \sum_{i=1}^r \lambda_i(A+B)
        \le\sum_{i=1}^r \lambda_i(A)+\sum_{i=1}^r \lambda_i(B).
    \end{equation}
\end{lemma}
\begin{proof}
    This is a direct corollary of~\cite[Theorem~1]{Fan}.
\end{proof}

\begin{lemma}\label{lem:gmb}
    Let $H$ be a graph with degree sequence
    $d_1,\ldots,d_n$, $m(H)$ edges, and
    Laplacian eigenvalues
    $\lambda_1(H)\ge\cdots\ge\lambda_n(H)=0$.  Then
    \begin{equation}
        \sum_{i=1}^r \lambda_i(H)
        \le \sum_{v\in V(H)}\min\{d_H(v),r\}
        \le m(H)+\frac{nr}{2}
    \end{equation}
    for $1\le r\le n$.
\end{lemma}
\begin{proof}
    The first inequality
    is the Grone--Merris conjecture~\cite{Grone},
    which is proved by Bai~\cite{Bai}.
    The second follows from $\min\{d,r\}\le(d+r)/2$
    and $\sum_v d_H(v)=2m(H)$.
\end{proof}

\begin{lemma}\label{lem:vertex-cover-spectral-bound}
    Let $H$ be a graph
    and let $X\subseteq V(H)$ be a vertex cover of $H$.
    Then, for $1\le r\le|V(H)|$,
    \begin{equation}
        \sum_{i=1}^r\lambda_i(L(H))\le m(H)+|X|r.
    \end{equation}
\end{lemma}
\begin{proof}
    By the first inequality in Lemma~\ref{lem:gmb},
    \begin{equation}
        \sum_{i=1}^r\lambda_i(L(H))
        \le\sum_{u\in V(H)}\min\{d_H(u),r\}.
    \end{equation}
    Since $X$ is a vertex cover,
    $V(H)\setminus X$ is independent, and hence
    \begin{equation}
        \sum_{u\notin X}d_H(u)\le m(H).
    \end{equation}
    Also
    $\sum_{u\in X}\min\{d_H(u),r\}\le|X|r$.
    The assertion follows.
\end{proof}

For the case $r=n$, the target inequality follows
from a simple estimate as soon as
the order of the graph is large compared with its diameter.

\begin{lemma}\label{lem:terminal-wiener}
    Let $G$ be a connected graph of order $n$ and diameter at most $D$.
    If $n\ge 3D$, then the target inequality holds for $r=n$.
\end{lemma}
\begin{proof}
    Since $U_n(G)=\tr(\DL(G))=2W(G)$,
    it is enough to show
    $W(G)\le\binom{n+2}{3}$.
    Indeed,
    \begin{equation}
        W(G)\le D\binom n2
        \le\frac n3\binom n2
        =\frac{n^2(n-1)}6
        \le\binom{n+2}{3}.
    \end{equation}
    This completes the proof.
\end{proof}

\section{Graphs with bounded diameter}\label{sec:3}
In this section we prove the target inequality
for all connected graphs
whose diameter is bounded by a fixed constant $D$,
provided that the order $n$ is sufficiently large in terms of $D$.
The main idea is to decompose the distance Laplacian matrix
into ordinary Laplacian matrices of graphs
recording which vertex pairs are at distance at least $s$,
and then reduce the target inequality to a numerical estimate.

For an integer $2\le s\le D$,
we define $Y_s=Y_s(G)$ to be a new graph on $V(G)$
whose edges are $\{u,v\}$ such that $d_G(u,v)\ge s$.  
Equivalently, $Y_s$ is the complement of the $(s-1)$-th power of $G$.
For convenience, we call it
the \emph{distance-at-least-$s$} graph.

We prove that $D(G)$
can be expressed in terms of $I$ (the identity matrix),
$J$ (the all-ones matrix)
and sums of $A(Y_s)$ (the adjacency matrix of $Y_s$);
while $\DL(G)$ can be expressed in terms of $I,\,J$
and sums of $L(Y_s)$ (the Laplacian matrix of $Y_s$).

\begin{lemma}\label{lem:layer-decomp}
    If $G$ is connected and $\diam(G)\le D$, then
    \begin{equation}
        D(G)=J-I+\sum_{s=2}^D A(Y_s),\qquad
        \DL(G)=nI-J+\sum_{s=2}^D L(Y_s),
    \end{equation}
    where $A(H)$ and $L(H)$
    denote the adjacency matrix and the Laplacian matrix
    of a graph $H$, respectively.
\end{lemma}
\begin{proof}
    Let $A_i$ be the \emph{$i$-th distance matrix} of $G$,
    as in~\cite[Section~2.3]{Williams},
    whose $(u,v)$-entry is $1$ when $d_G(u,v)=i$,
    and is $0$ otherwise.
    Then
    \begin{equation}
        J-I=\sum_{i=1}^D A_i,\qquad
        A(Y_s)=\sum_{i=s}^D A_i.
    \end{equation}
    Therefore
    \begin{equation}
        J-I+\sum_{s=2}^D A(Y_s)
        =\sum_{i=1}^D A_i+\sum_{s=2}^D\sum_{i=s}^D A_i
        =\sum_{i=1}^D iA_i
        =D(G).
    \end{equation}
    This is the first identity.

    For the second identity,
    denote by $\mathbf 1$ the all-ones vector.
    Then applying the linear map
    \begin{equation}
        M\mapsto\diag(M\mathbf 1)-M
    \end{equation}
    to the first identity gives the second one, since
    \begin{equation}
        \diag((J-I)\mathbf 1)-(J-I)=nI-J
    \end{equation}
    and $L(H)=\diag(A(H)\mathbf 1)-A(H)$ for every graph $H$.
\end{proof}

These graphs also give a convenient expression
for the Wiener index.

\begin{lemma}\label{lem:layer-wiener}
    If $G$ is connected and $\diam(G)\le D$, then
    \begin{equation}
        W(G)=\binom n2+\sum_{s=2}^D m(Y_s).
    \end{equation}
\end{lemma}
\begin{proof}
    Each vertex pair $\{u,v\}$ contributes one unit to $W(G)$,
    and then contributes one additional unit for every
    $s\ge 2$ not exceeding $d_G(u,v)$.  Thus
    \begin{equation}
        W(G)
        =\sum_{\{u,v\}\subseteq V(G)}1
        +\sum_{s=2}^D\left|\{\{u,v\}:d_G(u,v)\ge s\}\right|
        =\binom n2+\sum_{s=2}^D m(Y_s).
    \end{equation}
\end{proof}

The next lemma converts bounds
for the Laplacian eigenvalues
of the graphs $Y_s$ into the target inequality.

\begin{lemma}\label{lem:layer-sum-reduction}
    Let $G$ be connected, have order $n$,
    and satisfy $\diam(G)\le D$.
    Suppose that for some number $B$,
    \begin{equation}
        \sum_{s=2}^D\sum_{i=1}^r\lambda_i(L(Y_s))
        \le\sum_{s=2}^D m(Y_s)+Br
    \end{equation}
    holds for $1\le r\le n-1$.
    Then the target inequality holds for such $r$ once
    \begin{equation}
        \binom n2+\binom{r+2}{3}-(n+B)r\ge 0.
    \end{equation}
\end{lemma}
\begin{proof}
    By Lemma~\ref{lem:layer-decomp},
    $\DL(G)=nI-J+\sum_{s=2}^D L(Y_s)$.
    The matrix $nI-J$ has sum of its top $r$ eigenvalues
    equal to $rn$ for $r\le n-1$.
    Lemma~\ref{lem:ky-fan} then gives
    \begin{equation}
        U_r(G)
        \le rn+\sum_{s=2}^D\sum_{i=1}^r\lambda_i(L(Y_s))
        \le \sum_{s=2}^D m(Y_s)+(n+B)r.
    \end{equation}
    By Lemma~\ref{lem:layer-wiener},
    $W(G)=\binom n2+\sum_{s=2}^D m(Y_s)$.
    Therefore
    \begin{equation}
        \begin{aligned}
            W(G)-U_r(G)+\binom{r+2}{3}
            &\ge W(G)-\sum_{s=2}^D m(Y_s)-(n+B)r+\binom{r+2}{3} \\
            &=\binom n2+\binom{r+2}{3}-(n+B)r
            \ge 0,
        \end{aligned}
    \end{equation}
    which is the target inequality.
\end{proof}

By the lemma above, we can derive the main result
of this section.

\begin{theorem}
    Let $D\ge 2$.
    If $G$ is connected,
    has order $n$,
    satisfies $\diam(G)\le D$, and
    \begin{equation}\label{eq:fixed-diam-n-aspt}
        n\ge \left\lceil \frac49(D+1)^3\right\rceil,
    \end{equation}
    then the target inequality holds for every $1\le r\le n$.
\end{theorem}
\begin{proof}
    We first treat the range $1\le r\le n-1$.
    Lemma~\ref{lem:gmb} gives
    \begin{equation}
        \sum_{s=2}^D\sum_{i=1}^r\lambda_i(L(Y_s))
        \le \sum_{s=2}^D m(Y_s)+\frac{(D-1)nr}{2}.
    \end{equation}
    Write $c=(D+1)/2$.
    By Lemma~\ref{lem:layer-sum-reduction} with $B=(D-1)n/2$,
    the target inequality follows once
    \begin{equation}\label{eq:sec3-1}
        \binom n2+\binom{r+2}{3}-cnr
        =\frac{n(n-1)}{2}+\frac{r(r+1)(r+2)}{6}-cnr\ge 0
        \qquad (1\le r\le n-1).
    \end{equation}
    Multiplying both sides by $6$,
    the left side of~\eqref{eq:sec3-1} becomes
    \begin{align}
        P_D(n,r)&\coloneqq 3n(n-1)+r(r+1)(r+2)-6cnr \\
        &=3n(n-1-2cr)+r(r+1)(r+2) \\
        &=3n^2-3n+r^3+3r^2+2r-6cnr.
    \end{align}
    By~\eqref{eq:fixed-diam-n-aspt} and $D+1=2c$,
    we obtain
    \begin{equation}
        n\ge \frac49(D+1)^3=\frac{32}{9}c^3.
    \end{equation}

    First suppose $r<\sqrt n$.
    Since $D\ge 2$, we have $c\ge 3/2$.
    The function
    \begin{equation}
        \frac{32}{9}c^3-\frac{8\sqrt2}{3}c^{5/2}-1
    \end{equation}
    is increasing for $c\ge 3/2$
    and equals $11-6\sqrt3>0$ at $c=3/2$.
    Hence
    \begin{equation}
        n-1-2c\sqrt n
        \ge\frac{32}{9}c^3-1-2c\sqrt{\frac{32}{9}c^3}
        =\frac{32}{9}c^3-\frac{8\sqrt2}{3}c^{5/2}-1
        \ge 0.
    \end{equation}
    Therefore
    \begin{align}
        P_D(n,r)&=3n(n-1-2cr)+r(r+1)(r+2) \\
        &>3n(n-1-2c\sqrt{n})+r(r+1)(r+2)\ge 0,
    \end{align}
    which proves~\eqref{eq:sec3-1} for $r<\sqrt{n}$.
    Now suppose $r\ge\sqrt n$.
    Since the minimum of
    the function $x^3-Bx$ over $x\ge 0$
    is $-2B^{3/2}/(3\sqrt3)$, we have
    \begin{equation}\label{eq:sec3-2}
        r^3-6cnr
        \ge\frac{-2(6cn)^{3/2}}{3\sqrt{3}}
        =-4\sqrt2\,c^{3/2}n^{3/2}
    \end{equation}
    for every $r\ge 0$.
    Also $n\ge 32c^3/9$ implies
    \begin{equation}\label{eq:sec3-3}
        n^4\ge\frac{32c^3n^3}{9}
        \quad\implies\quad
        n^2-\frac{4\sqrt{2}}{3}c^{3/2}n^{3/2}\ge 0
        \quad\implies\quad
        3n^2-4\sqrt{2}\,c^{3/2}n^{3/2}\ge 0.
    \end{equation}
    Adding~\eqref{eq:sec3-2} and~\eqref{eq:sec3-3} gives
    \begin{equation}
        r^3+3n^2-6cnr\ge 0
        \quad\implies\quad
        3n^2-6cnr\ge -r^3.
    \end{equation}
    Thus
    \begin{equation}
        P_D(n,r)=(3n^2-6cnr)+(-3n+r^3+3r^2+2r)
        \ge -3n+3r^2+2r
        \ge 2r
        \ge 0.
    \end{equation}
    This proves~\eqref{eq:sec3-1} for $r\ge\sqrt{n}$,
    and hence the target inequality,
    for every $1\le r\le n-1$.

    Finally,~\eqref{eq:fixed-diam-n-aspt}
    implies $n\ge 3D$ when $D\ge 2$,
    so the case $r=n$ directly follows
    from Lemma~\ref{lem:terminal-wiener}.
\end{proof}

By taking $D=3,4,5$, we have the following corollary.

\begin{corollary}\label{cor:diam-three-four}
    The target inequality holds for every connected graph $G$
    of order $n$ in each of the following cases:
    \begin{enumerate}
        \item $\diam(G)\le 3$ and $n\ge 29$;
        \item $\diam(G)\le 4$ and $n\ge 56$;
        \item $\diam(G)\le 5$ and $n\ge 96$.
    \end{enumerate}
\end{corollary}

Corollary~\ref{cor:diam-three-four}
considerably strengthens the diameter-$3$ and diameter-$4$
results of Mushtaq~et~al.~\cite{Mushtaq}.
Their diameter-$3$ results apply to
the special class $\Gamma_n$ for $n\ge 95$,
while our corollary covers all connected graphs
of diameter at most $3$ once $n\ge 29$.
For diameter $4$, they proved the inequality
for sun graphs and partial sun graphs
under additional parameter restrictions,
whereas our corollary applies to every connected graph
of diameter at most $4$ once $n\ge 56$.

\section{Diameter-\texorpdfstring{$2$}{2} graphs}\label{sec:4}
In this section,
we fully determine the diameter-$2$ case.
The key technique is that, for such graphs,
the distance Laplacian matrix can be written
in terms of the Laplacian matrices of their graph complements.
This allows us to reduce most values of $r$
to known cases of the original Brouwer's conjecture
and then handle the remaining small cases directly.

We first record the validity of the original Brouwer's conjecture
for several values of $r$,
which will be useful later.

\begin{lemma}\label{lem:original-brouwer-known}
    Let $H$ be a graph of order $n$, size $m$,
    and ordinary Laplacian eigenvalues
    $\lambda_1(H)\ge\cdots\ge\lambda_n(H)=0$.  Put
    $S_r(H)=\sum_{i=1}^r\lambda_i(H)$. Then
    \begin{equation}
        S_r(H)\le m+\binom{r+1}{2}
    \end{equation}
    holds for $r\in\{2,3,n-1,n\}$.
    Moreover, if $n\ge 3$, then this inequality
    also holds for $r=n-2$;
    if $n\ge 5$, then this inequality
    also holds for $r=n-3$.
\end{lemma}
\begin{proof}
    The case $r=2$ is proved by Haemers~\cite{Haemers}.
    The case $r=3$ is proved by Wang~et~al.~\cite{Wang}.
    The cases $r=n-2,\,n-3$ then follow from
    Chen~\cite[Corollary~3.2 and~3.3]{Chen2019}.
    For $r=n-1$, since $\lambda_n(H)=0$,
    we have $S_{n-1}(H)=2m$,
    and the desired inequality follows from
    $2m\le m+\binom n2$.
    For $r=n$, we also have $S_n(H)=2m$,
    and the desired inequality follows from
    $2m\le m+\binom{n+1}{2}$.
\end{proof}

For diameter-$2$ graphs, the distance Laplacian matrix
can be expressed through the Laplacian matrix
of their complement graphs.

\begin{lemma}\label{lem:sec4-transfer}
    Let $G$ be a connected graph of order $n$ and diameter $2$,
    and put $H=\overline G$, the complement of $G$.
    Then
    \begin{equation}
        D(G)=J-I+A(H),\qquad
        \DL(G)=nI-J+L(H).
    \end{equation}
    Consequently,
    \begin{equation}
        \partial_i^L(G)=n+\lambda_i(H)\quad(1\le i\le n-1),\qquad
        \partial_n^L(G)=0.
    \end{equation}
\end{lemma}
\begin{proof}
    In a graph with diameter $2$,
    nonadjacent vertex pairs are exactly
    the pairs at distance $2$. Hence
    \begin{equation}
        D(G)=A(G)+2A(H)=J-I-A(H)+2A(H)=J-I+A(H).
    \end{equation}
    Also
    \begin{equation}
        \Tr_G(v)=\deg_G(v)+2\deg_H(v)=n-1+\deg_H(v),
    \end{equation}
    and so
    \begin{equation}
        \DL(G)=\diag(n-1+\deg_H(v))-(J-I+A(H))=nI-J+L(H).
    \end{equation}
    Since $J$ vanishes on $\mathbf 1^\perp$,
    the subspace $\mathbf 1^\perp$ is invariant under $\DL(G)$,
    and there $\DL(G)$ acts as $nI+L(H)$.
    On the span of $\mathbf 1$, it acts as zero.
    The spectral formula then follows.
\end{proof}

This also gives the relations between $U_r$ and $S_r$.

\begin{lemma}\label{lem:sec4-sums}
    Let $G$ be a connected graph of order $n$ and diameter $2$,
    and put $H=\overline G$. Then
    \begin{equation}
        U_r(G)=rn+S_r(H)\qquad(1\le r\le n-1).
    \end{equation}
\end{lemma}
\begin{proof}
    This follows directly from Lemma~\ref{lem:sec4-transfer}.
\end{proof}

The Wiener index
can be expressed through the graph order and
the size of the complement graph.

\begin{lemma}\label{lem:sec4-wiener}
    Let $G$ be a connected graph of order $n$ and diameter $2$,
    and put $H=\overline G$ and $m=m(H)$. Then
    \begin{equation}
        W(G)=\binom n2+m.
    \end{equation}
\end{lemma}
\begin{proof}
    The $m$ edges of $H$ are exactly the nonedges of $G$,
    hence exactly the vertex pairs at distance $2$ in $G$.
    Therefore
    \begin{equation}
        W(G)=\left(\binom n2-m\right)+2m=\binom n2+m,
    \end{equation}
    completing the proof.
\end{proof}

We will convert a bound for $S_r(\overline G)$
into the target inequality by the following lemma.

\begin{lemma}\label{lem:sec4-brouwer-reduction}
    Let $G$ be a connected graph of order $n$ and diameter $2$,
    and put $H=\overline G$ and $m=m(H)$.
    Let $1\le r\le n-1$. If
    \begin{equation}
        S_r(H)\le m+T,
    \end{equation}
    then the target inequality holds for such an $r$ once
    \begin{equation}\label{eq:sec4-brouwer-reduction-gap}
        \binom n2+\binom{r+2}{3}-rn-T\ge 0.
    \end{equation}
\end{lemma}
\begin{proof}
    Lemma~\ref{lem:sec4-sums} gives
    $U_r(G)=rn+S_r(H)$
    and Lemma~\ref{lem:sec4-wiener} gives $W(G)=\binom n2+m$.
    Hence if~\eqref{eq:sec4-brouwer-reduction-gap} holds, we have
    \begin{equation}
        U_r(G)
        \le rn+m+T
        \le m+\binom n2+\binom{r+2}{3}
        =W(G)+\binom{r+2}{3},
    \end{equation}
    which is the target inequality.
\end{proof}

The next theorem handles the range $2\le r\le n-1$
once the order is at least $8$.

\begin{theorem}\label{thm:sec4-middle-large}
    Let $G$ be a connected diameter-$2$ graph of order $n\ge 8$.
    Then the target inequality holds for every $2\le r\le n-1$.
\end{theorem}
\begin{proof}
    Put $H=\overline G$ and $m=m(H)$.
    Lemma~\ref{lem:gmb} gives
    \begin{equation}
        S_r(H)\le m+\frac{nr}{2}.
    \end{equation}
    Therefore by Lemma~\ref{lem:sec4-brouwer-reduction},
    it is enough to prove
    \begin{equation}
        \binom n2+\binom{r+2}{3}-\frac{3nr}{2}\ge 0.
    \end{equation}
    After multiplying $6$, this becomes
    \begin{equation}\label{eq:sec4-1}
        F(n,r)\coloneqq r^3+3r^2+(2-9n)r+3n^2-3n\ge 0.
    \end{equation}

    For $n\ge 9$, write $x=r+1$. Then
    \begin{equation}
        F(n,r)=3n^2+6n+x^3-(9n+1)x.
    \end{equation}
    Since the minimum of $x^3-Bx$ over $x\ge 0$ is
    $-2B^{3/2}/(3\sqrt3)$, it is enough to show
    \begin{equation}
        H(n)\coloneqq
        3n^2+6n-\frac{2(9n+1)^{3/2}}{3\sqrt3}\ge 0.
    \end{equation}
    In fact,
    the function $H$ is increasing for $n\ge 9$, because
    \begin{equation}
        H'(n)=6n+6-3\sqrt3(9n+1)^{1/2}>0,
    \end{equation}
    which follows from
    \begin{equation}
        (2n+2)^2>3(9n+1),
    \end{equation}
    which holds for $n\ge 9$.
    Also, $H(9)=297-164\sqrt{246}/9>0$.
    Therefore~\eqref{eq:sec4-1} is true for all $n\ge 9$.
    For $n=8$, direct calculation gives
    \begin{equation}
        \begin{array}{c|cccccc}
        r&2&3&4&5&6&7\\ \hline
        F(8,r)&48&12&0&18&72&168
        \end{array}
    \end{equation}
    Therefore~\eqref{eq:sec4-1} also holds.
\end{proof}

The case $2\le r\le n-1$ with $n\in\{6,7\}$
is covered by Lemma~\ref{lem:original-brouwer-known}.

\begin{theorem}\label{theorem:sec4-middle-six-seven}
    Let $G$ be a connected diameter-$2$ graph of order $n\in\{6,7\}$.
    Then the target inequality holds for every $2\le r\le n-1$.
\end{theorem}
\begin{proof}
    Put $H=\overline G$ and $m=m(H)$.
    By Lemma~\ref{lem:original-brouwer-known},
    \begin{equation}
        S_r(H)\le m+\binom{r+1}{2}
        \qquad(2\le r\le n-1)
    \end{equation}
    is true for $n\in\{6,7\}$.
    Therefore by Lemma~\ref{lem:sec4-brouwer-reduction},
    it remains to check
    \begin{equation}
        \binom n2+\binom{r+2}{3}-rn-\binom{r+1}{2}\ge 0.
    \end{equation}
    In fact, direct calculation gives
    \begin{equation}
        \begin{array}{c|cccc}
        \raisebox{-0.5\normalbaselineskip}{$n=6$}
        & r=2&r=3&r=4&r=5\\ \cline{2-5}
        &4&1&1&5
        \end{array}
        \qquad
        \begin{array}{c|ccccc}
        \raisebox{-0.5\normalbaselineskip}{$n=7$}
        & r=2&r=3&r=4&r=5&r=6\\ \cline{2-6}
        &8&4&3&6&14
        \end{array}
    \end{equation}
    Therefore the target inequality holds.
\end{proof}

We now combine the preceding theorems
with an analysis of the small-order graphs
to obtain the full diameter-$2$ graph classification.

\begin{theorem}\label{thm:sec4-exact}
    Let $G$ be a connected graph of diameter $2$.
    Then the target inequality holds
    for every valid $r$,
    except for $K_{1,3}$ at $r=2$ and $K_{1,4}$ at $r=3$
    (Figure~\ref{fig:1}).
\end{theorem}

\begin{figure}[htb]
    \centering
    \begin{tikzpicture}[
        every node/.style={circle,fill,inner sep=1.6pt},
        edge/.style={line width=0.5pt}
    ]
        \begin{scope}
            \node (c3) at (0,0) {};
            \foreach \i/\a in {1/90,2/210,3/330} {
                \node (l3\i) at ({cos(\a)},{sin(\a)}) {};
                \draw[edge] (c3) -- (l3\i);
            }
        \end{scope}

        \begin{scope}[xshift=4.2cm]
            \node (c4) at (0,0) {};
            \foreach \i/\a in {1/45,2/135,3/225,4/315} {
                \node (l4\i) at ({cos(\a)},{sin(\a)}) {};
                \draw[edge] (c4) -- (l4\i);
            }
        \end{scope}
    \end{tikzpicture}
    \caption{Illustrations of $K_{1,3}$ and $K_{1,4}$, respectively.}
    \label{fig:1}
\end{figure}
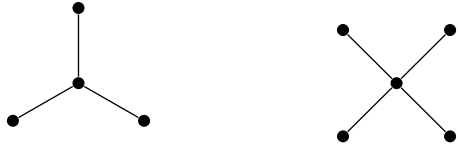

\begin{proof}
    Put $H=\overline G$ and $m=m(H)$ and suppose $n\ge 6$.

    For $r=1$, the bound $\lambda_1(H)\le n$ gives $U_1(G)\le 2n$.
    Since a diameter-$2$ graph is not complete,
    we have $m\ge 1$. Hence for $n\ge 5$, we have
    \begin{equation}
        W(G)+\binom{1+2}{3}-U_1(G)
        \ge\left(\binom n2+m\right)+1-2n
        \ge\binom n2+2-2n
        =\frac{(n-1)(n-4)}2>0,
    \end{equation}
    so this proves the $r=1$ case.

    For $2\le r\le n-1$, the assertion follows from
    Theorem~\ref{thm:sec4-middle-large} if $n\ge 8$,
    and from Theorem~\ref{theorem:sec4-middle-six-seven}
    if $n\in\{6,7\}$.

    For $r=n$, the trivial Wiener bound gives
    \begin{equation}
        W(G)\le 2\binom n2\le\binom{n+2}{3}=\binom{r+2}{3}.
    \end{equation}
    Hence $U_n(G)=2W(G)\le W(G)+\binom{r+2}{3}$.
    
    It remains only to discuss $n\le 5$.
    For $n=3$, the only diameter-$2$ graph is
    $P_3=K_{1,2}$,
    whose distance Laplacian spectrum is $\{5,3,0\}$.
    The target inequality follows immediately.

    Now let $n=4$, and again put $H=\overline G$ and $m=m(H)\ge 1$.
    For $r=1$, the bound $\lambda_1(H)\le 4$ gives
    \begin{equation}
        U_1(G)=4+\lambda_1(H)
        \le 8
        \le\binom 42+m+1
        =W(G)+1
        =W(G)+\binom{1+2}{3}.
    \end{equation}
    For $r=3$, since $\lambda_4(H)=0$,
    we have $S_3(H)=\sum_{i=1}^3\lambda_i(H)=\sum_{i=1}^4\lambda_i(H)=2m$.
    Since $G$ is connected and has order $4$,
    $G$ has at least $3$ edges,
    and so $m\le\binom 42-3=3$. Hence
    \begin{equation}
        U_3(G)=12+S_3(H)
        =12+2m
        <16+m
        =W(G)+\binom{3+2}{3}.
    \end{equation}
    The case $r=4$ follows from $U_4(G)=2W(G)$ and
    $W(G)\le 2\binom42<\binom{4+2}{3}$.
    It remains only to consider $r=2$.
    By Lemma~\ref{lem:sec4-brouwer-reduction} with $T=2$,
    it is enough to prove $S_2(H)\le m+2$.
    If $m\le 2$, this follows from $S_2(H)\le 2m$.
    If $m\ge 3$, then $G=\overline H$ has at most $3$ edges.
    Since $G$ is connected and has order $4$, it must be a tree;
    diameter $2$ then forces $G=K_{1,3}$,
    or equivalently $H=K_3\cup K_1$.

    Let $n=5$.
    The cases $r=1,2,4$ follow from the preceding arguments:
    for $r=1$ use $\lambda_1(H)\le 5$ and $m\ge 1$,
    while for $r=2,4$ use Lemma~\ref{lem:original-brouwer-known}.
    The case $r=5$ follows from $U_5(G)=2W(G)$ and
    $W(G)\le 2\binom52<\binom73$.
    Thus only $r=3$ remains.
    By Lemma~\ref{lem:sec4-brouwer-reduction} with $T=5$,
    it is enough to prove $S_3(H)\le m+5$.
    If $m\le 5$, then this follows from $S_3(H)\le 2m$.
    If $m\ge 6$, then $G=\overline H$ has at most $4$ edges.
    Since $G$ is connected and has order $5$, it must be a tree;
    diameter $2$ then forces $G=K_{1,4}$,
    or equivalently $H=K_4\cup K_1$.

    Thus the only remaining cases are
    $K_{1,3}$ at $r=2$ and $K_{1,4}$ at $r=3$.
    We now verify directly that these are indeed exceptions.
    In general, for the star $K_{1,t}$,
    \begin{equation}
        \operatorname{Spec}(\DL(K_{1,t}))
        =\{(2t+1)^{(t-1)},\,t+1,\,0\},
        \qquad
        W(K_{1,t})=t^2.
    \end{equation}
    The former expression comes from
    the proof of~\cite[Theorem~3.6]{Zhou},
    while the latter expression can be verified
    by direct calculation.
    Hence, for $K_{1,3}$,
    \begin{equation}
        U_2(K_{1,3})=7+7=14
        >13=9+\binom{4}{3}
        =W(K_{1,3})+\binom{2+2}{3}.
    \end{equation}
    For $K_{1,4}$,
    \begin{equation}
        U_3(K_{1,4})=9+9+9=27
        >26=16+\binom{5}{3}
        =W(K_{1,4})+\binom{3+2}{3}.
    \end{equation}
    Therefore these two cases are precisely the exceptions.
\end{proof}

For connected diameter-$2$ graphs,
Zhou~et~al.~\cite{Zhou}
proved the target inequality for several classes of such graphs,
namely, sufficiently dense graphs,
graphs whose complements satisfy the original Brouwer's conjecture,
and graphs with maximum degree $n-1$ or $n-2$
where $n\ge 21$ or $n\ge 26$, respectively.
Here, Theorem~\ref{thm:sec4-exact}
gives a complete solution for diameter-$2$ graphs.

\section{Graphs with maximum degree \texorpdfstring{$n-k$}{n-k}}\label{sec:5}
In this section we prove the target inequality
for graphs whose maximum degree is close to the order.
If $\Delta(G)=n-k$,
then a vertex of maximum degree leaves only $k-1$ vertices
outside its closed neighborhood.
This small vertex subset controls
the spectra of the distance-at-least-$s$ graphs for $s\ge 3$.

We first record the precise form of this observation.

\begin{lemma}\label{lem:max-degree-n-minus-k-structure}
    Let $k\ge 2$,
    and let $G$ be a connected graph of order $n$
    with maximum degree $\Delta(G)=n-k$.
    Let $v$ be a vertex of degree $n-k$
    and $X=V(G)\setminus N[v]$.
    Then $|X|=k-1$, $\diam(G)\le k+1$, and,
    for every $3\le s\le k+1$,
    every edge of $Y_s$ has at least one endpoint in $X$.
    Equivalently, $X$ is a vertex cover of $Y_s$.
    Consequently, for $1\le r\le n-1$,
    \begin{equation}
        \sum_{i=1}^r \lambda_i(L(Y_s))\le m(Y_s)+(k-1)r
        \qquad(3\le s\le k+1).
    \end{equation}
\end{lemma}
\begin{proof}
    Put $C=N[v]$.
    Then $|X|=k-1$, and any two vertices of $C$
    are within distance at most $2$.

    Contract $C$ to one vertex $c$.
    Since $G$ is connected,
    the contracted graph is also connected and has $k$ vertices.
    Hence the distance between any two of its vertices
    is at most $k-1$.
    Returning to $G$,
    a path that passes through the contracted vertex $c$
    may be expanded inside $C$ at cost at most $2$ edges.
    It follows that any two vertices of $G$
    are within distance at most $k+1$.
    Thus $\diam(G)\le k+1$.

    Since all pairs contained in $C$ have distance at most $2$,
    every edge of $Y_s$ with $s\ge 3$
    is incident with a vertex of $X$.
    Thus $X$ is a vertex cover of $Y_s$.
    The bound then follows from
    Lemma~\ref{lem:vertex-cover-spectral-bound}.
\end{proof}

We now derive the main result of this section.

\begin{theorem}\label{thm:max-degree-n-minus-k}
    Let $k\ge 2$, and define
    \begin{equation}
        N(k)=
        \begin{cases}
            10, & k=2,\\
            \lceil 5(k-1)^{3/2}\rceil, & k\ge 3.
        \end{cases}
    \end{equation}
    If $G$ is a connected graph of order $n\ge N(k)$
    and $\Delta(G)=n-k$,
    then the target inequality holds for every $1\le r\le n$.
\end{theorem}
\begin{proof}
    Write $a=(k-1)^2$.
    By Lemma~\ref{lem:max-degree-n-minus-k-structure},
    $\diam(G)\le k+1$.
    Hence, by Lemma~\ref{lem:layer-decomp},
    \begin{equation}
    \DL(G)=nI-J+\sum_{s=2}^{k+1}L(Y_s).
    \end{equation}

    For $1\le r\le n-1$, Lemmas~\ref{lem:gmb} and
    \ref{lem:max-degree-n-minus-k-structure} give
    \begin{align}
        \sum_{s=2}^{k+1}\sum_{i=1}^r\lambda_i(L(Y_s))
        &\le m(Y_2)+\frac{nr}{2}
        +\sum_{s=3}^{k+1}\left(m(Y_s)+(k-1)r\right)  \\
        &=\sum_{s=2}^{k+1}m(Y_s)+\frac{nr}{2}+ar.
    \end{align}
    Thus Lemma~\ref{lem:layer-sum-reduction},
    with $B=n/2+a$, reduces the target inequality to
    \begin{equation}\label{eq:max-degree-n-minus-k-gap}
        \binom n2+\binom{r+2}{3}-\frac{3nr}{2}-ar\ge 0
        \qquad(1\le r\le n-1).
    \end{equation}
    After multiplying $6$, the left-hand side becomes
    \begin{equation}
        F_a(n,r)=
        3n^2-3n+r^3+3r^2+(2-6a)r-9nr.
    \end{equation}

    First suppose $k=2,\,a=1$ and $n\ge 10$.
    For fixed $r$, the quadratic $F_1(n,r)$
    reaches its minimum at $n=(3r+1)/2$.
    Hence, for $r\ge 6$,
    \begin{equation}
        F_1(n,r)\ge F_1\left(\frac{3r+1}{2},r\right)
        =r^3-\frac{15}{4}r^2-\frac{17}{2}r-\frac34
        =r\left(r\left(r-\frac{15}{4}\right)-\frac{17}{2}\right)-\frac34>0,
    \end{equation}
    where the last expression is increasing
    and positive for $r\ge 6$.
    For $1\le r\le 5$, the polynomial $F_1(n,r)$ is increasing in $n$ for
    $n\ge 10$, and
    \begin{equation}
        F_1(10,r)=r^3+3r^2-94r+270,
    \end{equation}
    whose derivative $3r^2+6r-94$ is negative on $1\le r\le 5$, so
    $F_1(10,r)\ge F_1(10,5)=0$.
    Thus~\eqref{eq:max-degree-n-minus-k-gap} holds
    for $n\ge 10$ when $k=2$.

    Now suppose $k\ge 3$.
    Write $b=k-1$, so $b\ge 2$ and $a=b^2$.
    Put $x=r+1$. Since we have the identity
    \begin{equation}
        r^3+3r^2+2r=x^3-x,
    \end{equation}
    we obtain
    \begin{equation}
        \begin{aligned}
        F_a(n,r)
        &=3n^2-3n+r^3+3r^2+2r-(9n+6b^2)r\\
        &=3n^2+6n+6b^2+x^3-(9n+6b^2+1)x.
        \end{aligned}
    \end{equation}
    For $B>0$, the minimum of $x^3-Bx$
    over $x\ge 0$ is $-2B^{3/2}/(3\sqrt3)$.
    Thus it is enough to show that
    \begin{equation}\label{eq:sec5-1}
        H(n)\coloneqq
        3n^2+6n+6b^2
        -\frac{2(9n+6b^2+1)^{3/2}}{3\sqrt3}\ge 0.
    \end{equation}
    We claim that
    the function $H$ is increasing for $n\ge 5b^{3/2}$.
    Indeed,
    \begin{equation}
        H'(n)=6n+6-3\sqrt3(9n+6b^2+1)^{1/2},
    \end{equation}
    and
    \begin{equation}
        H'(n)>0
        \quad\iff\quad
        (2n+2)^2-3(9n+6b^2+1)
        =4n^2-19n-18b^2+1>0,
    \end{equation}
    which is increasing in $n$ for $n\ge 19/8$,
    and at $n=5b^{3/2}$ it is
    \begin{equation}
        100b^3-18b^2-95b^{3/2}+1
        >(100-9-34)b^3+1>0,
    \end{equation}
    where the second-to-last inequality follows from
    \begin{equation}
        18b^2\le 9b^3,\quad
        95b^{3/2}\le\frac{95}{2^{3/2}}b^3<34b^3
        \qquad(b\ge 2).
    \end{equation}

    If $k=3,\,b=2$, then the assumption gives $n\ge 15$, and
    \begin{equation}
        H(15)=789-\frac{1280\sqrt{30}}9>0.
    \end{equation}
    Therefore~\eqref{eq:max-degree-n-minus-k-gap} holds.
    
    If $k\ge 4,\,b\ge 3$,
    then evaluate~\eqref{eq:sec5-1} at $n=5b^{3/2}$.  Since
    \begin{equation}
    9\cdot5b^{3/2}+6b^2+1\le 33b^2
    \le(26+6+1)b^2=33b^2
    \qquad(b\ge 3),
    \end{equation}
    we obtain
    \begin{equation}
    H(5b^{3/2})
    \ge 75b^3+30b^{3/2}+6b^2
    -\frac{2(33b^2)^{3/2}}{3\sqrt3}
    =(75-22\sqrt{11})b^3+6b^2+30b^{3/2}>0.
    \end{equation}
    This proves~\eqref{eq:max-degree-n-minus-k-gap} for all $k\ge 2$.

    For $r=n$, the assumption $n\ge N(k)$ implies $n\ge 3(k+1)$.
    Lemma~\ref{lem:terminal-wiener} therefore applies.
\end{proof}

Theorem~\ref{thm:max-degree-n-minus-k}
substantially improves the maximum-degree result
of Mushtaq~et~al.~\cite{Mushtaq}.
They proved the target inequality for graphs with
$\Delta(G)=n-2$ under the assumption $n\ge 95$.
Theorem~\ref{thm:max-degree-n-minus-k}
applies to the same class of graphs already for $n\ge 10$.
Moreover, our result is not restricted to the case
$\Delta(G)=n-2$, as for every fixed $k\ge 3$,
it applies to all connected graphs with
$\Delta(G)=n-k$ once $n\ge\lceil5(k-1)^{3/2}\rceil$.

\section{Conclusion}
We have proved the target inequality for three kinds of graphs:
connected graphs of bounded diameter and sufficiently large order,
all diameter-$2$ graphs except for the two stated star exceptions,
and graphs whose maximum degree is $n-k$
once $n$ is large enough in terms of $k$.

The common tool is the decomposition
of the distance Laplacian matrices
into Laplacian matrices of distance-at-least-$s$ graphs,
which reduces the problem
to familiar Laplacian eigenvalue estimates
and elementary numerical inequalities.

However, it would be interesting to
further sharpen the order thresholds
in the bounded-diameter and maximum-degree cases.

\bibliographystyle{amsplain-with-doi}
\bibliography{references}

@article{Fan,
  author  = {Ky Fan},
  title   = {On a Theorem of {Weyl} Concerning Eigenvalues of Linear Transformations {I}},
  journal = {Proceedings of the National Academy of Sciences},
  volume  = {35},
  number  = {11},
  pages   = {652-655},
  year    = {1949},
  doi     = {10.1073/pnas.35.11.652}
}

@article{Grone,
  author  = {Grone, Robert and Merris, Russell},
  title   = {The {Laplacian} Spectrum of a Graph {II}},
  journal = {SIAM Journal on Discrete Mathematics},
  volume  = {7},
  number  = {2},
  pages   = {221-229},
  year    = {1994},
  doi     = {10.1137/S0895480191222653}
}

@article{Bai,
  author  = {Bai, Hua},
  title   = {The {Grone}-{Merris} conjecture},
  journal = {Transactions of the American Mathematical Society},
  issn    = {0002-9947},
  volume  = {363},
  number  = {8},
  pages   = {4463--4474},
  year    = {2011},
  doi     = {10.1090/S0002-9947-2011-05393-6}
}

@book{Brouwer,
  title     = {Spectra of graphs},
  author    = {Brouwer, Andries E. and Haemers, Willem H.},
  year      = {2012},
  publisher = {Springer},
  address   = {New York},
  doi       = {10.1007/978-1-4614-1939-6}
}

@article{Haemers,
  author  = {Haemers, Willem H. and Mohammadian, Ali and Tayfeh-Rezaie, Behruz},
  title   = {On the sum of {Laplacian} eigenvalues of graphs},
  journal = {Linear Algebra and its Applications},
  issn    = {0024-3795},
  volume  = {432},
  number  = {9},
  pages   = {2214--2221},
  year    = {2010},
  doi     = {10.1016/j.laa.2009.03.038}
}

@article{Wang,
  author  = {Wang, Ke and Lin, Zhen and Zhang, Shumin and Ye, Chengfu},
  title   = {A proof of {Brouwer}'s conjecture for {$k=3$}},
  journal = {Linear Algebra and its Applications},
  issn    = {0024-3795},
  volume  = {736},
  pages   = {189--213},
  year    = {2026},
  doi     = {10.1016/j.laa.2026.01.026}
}

@article{Ganie,
  author  = {Ganie, Hilal Ahmad and Pirzada, Shariefuddin and Rather, Bilal Ahmad and Trevisan, Vilmar},
  title   = {Further developments on {Brouwer}'s conjecture for the sum of {Laplacian} eigenvalues of graphs},
  journal = {Linear Algebra and its Applications},
  issn    = {0024-3795},
  volume  = {588},
  pages   = {1--18},
  year    = {2020},
  doi     = {10.1016/j.laa.2019.11.020}
}

@article{Du,
  title   = {Upper bounds for the sum of {Laplacian} eigenvalues of graphs},
  journal = {Linear Algebra and its Applications},
  volume  = {436},
  number  = {9},
  pages   = {3672-3683},
  year    = {2012},
  issn    = {0024-3795},
  doi     = {10.1016/j.laa.2012.01.007},
  author  = {Zhibin Du and Bo Zhou}
}

@article{Chen2018,
  author  = {Chen, Xiaodan},
  title   = {Improved results on {Brouwer}'s conjecture for sum of the {Laplacian} eigenvalues of a graph},
  journal = {Linear Algebra and its Applications},
  issn    = {0024-3795},
  volume  = {557},
  pages   = {327--338},
  year    = {2018},
  doi     = {10.1016/j.laa.2018.08.003}
}

@article{Chen2019,
  author  = {Chen, Xiaodan},
  title   = {On {Brouwer}'s conjecture for the sum of {$k$} largest {Laplacian} eigenvalues of graphs},
  journal = {Linear Algebra and its Applications},
  issn    = {0024-3795},
  volume  = {578},
  pages   = {402--410},
  year    = {2019},
  doi     = {10.1016/j.laa.2019.05.029}
}

@article{Aouchiche2013,
  author  = {Aouchiche, Mustapha and Hansen, Pierre},
  title   = {Two {Laplacians} for the distance matrix of a graph},
  journal = {Linear Algebra and its Applications},
  issn    = {0024-3795},
  volume  = {439},
  number  = {1},
  pages   = {21--33},
  year    = {2013},
  doi     = {10.1016/j.laa.2013.02.030}
}

@article{Aouchiche2014,
  author  = {Aouchiche, Mustapha and Hansen, Pierre},
  title   = {Some properties of the distance {Laplacian} eigenvalues of a graph},
  journal = {Czechoslovak Mathematical Journal},
  year    = {2014},
  volume  = {64},
  number  = {3},
  pages   = {751--761},
  doi     = {10.1007/s10587-014-0129-2}
}

@article{Nath,
  title   = {On the distance {Laplacian} spectra of graphs},
  journal = {Linear Algebra and its Applications},
  volume  = {460},
  pages   = {97-110},
  year    = {2014},
  issn    = {0024-3795},
  doi     = {10.1016/j.laa.2014.07.025},
  author  = {Milan Nath and Somnath Paul}
}

@article{Lin,
  title   = {On the distance and distance {Laplacian} eigenvalues of graphs},
  journal = {Linear Algebra and its Applications},
  volume  = {492},
  pages   = {128-135},
  year    = {2016},
  issn    = {0024-3795},
  doi     = {10.1016/j.laa.2015.11.014},
  author  = {Huiqiu Lin and Baoyindureng Wu and Yingying Chen and Jinlong Shu}
}

@article{Fernandes,
  title   = {Multiplicities of distance {Laplacian} eigenvalues and forbidden subgraphs},
  journal = {Linear Algebra and its Applications},
  volume  = {541},
  pages   = {81-93},
  year    = {2018},
  issn    = {0024-3795},
  doi     = {10.1016/j.laa.2017.11.031},
  author  = {Fernandes, Ros{\'a}rio and de Freitas, Maria Aguieiras A. and da Silva, Celso M. Jr. and del-Vecchio, Renata Raposo}
}

@article{Brimkov,
  author  = {Brimkov, Boris and Duna, Ken and Hogben, Leslie and Lorenzen, Kate and Reinhart, Carolyn and Song, Sung-Yell and Yarrow, Mark},
  title   = {Graphs that are cospectral for the distance {Laplacian}},
  journal = {The Electronic Journal of Linear Algebra},
  issn    = {1081-3810},
  volume  = {36},
  pages   = {334--351},
  year    = {2020},
  doi     = {10.13001/ela.2020.4941}
}

@article{Rather,
  title   = {Distance {Laplacian} spectra of graphs: A survey},
  journal = {Discrete Applied Mathematics},
  volume  = {361},
  pages   = {136-195},
  year    = {2025},
  issn    = {0166-218X},
  doi     = {10.1016/j.dam.2024.10.001},
  author  = {Bilal Ahmad Rather and Mustapha Aouchiche}
}

@article{Pirzada,
  author  = {Pirzada, Shariefuddin and Khan, Saleem},
  title   = {On the sum of distance {Laplacian} eigenvalues of graphs},
  journal = {Tamkang Journal of Mathematics},
  issn    = {0049-2930},
  volume  = {54},
  number  = {1},
  pages   = {83--91},
  year    = {2023},
  doi     = {10.5556/j.tkjm.54.2023.4120}
}

@article{Zhou,
  author  = {Zhou, Yuwei and Wang, Ligong and Chai, Yirui},
  title   = {{Brouwer} type conjecture for the eigenvalues of distance {Laplacian} matrix of a graph},
  journal = {Computational and Applied Mathematics},
  issn    = {2238-3603},
  volume  = {44},
  number  = {3},
  pages   = {14},
  year    = {2025},
  doi     = {10.1007/s40314-025-03095-0}
}

@article{Mushtaq,
  author  = {Mushtaq, Ummer and Pirzada, Shariefuddin and Khan, Saleem},
  title   = {On the sum of the eigenvalues of the distance {Laplacian} matrix of graphs with diameter three and four},
  journal = {Revista de la Uni{\'o}n Matem{\'a}tica Argentina},
  issn    = {0041-6932},
  volume  = {69},
  number  = {1},
  pages   = {193--202},
  year    = {2026},
  doi     = {10.33044/revuma.5147}
}

@article{Williams,
  author  = {Williams, Gerald},
  title   = {{Smith} forms for adjacency matrices of circulant graphs},
  journal = {Linear Algebra and its Applications},
  issn    = {0024-3795},
  volume  = {443},
  pages   = {21--33},
  year    = {2014},
  doi     = {10.1016/j.laa.2013.11.006}
}

\end{document}